\documentclass{amsart}

\numberwithin{equation}{section}
\usepackage{graphicx}
\usepackage{amssymb}
\usepackage{amsmath}
\usepackage{setspace}
\usepackage{color}
\allowdisplaybreaks

\usepackage{hyperref}
\usepackage{cleveref}

\newcommand{\Sub}{\mathrm{Sub}}

\newcommand{\mcal}{\mathcal}

\newcommand{\N}{\mathbb N}

\newcommand{\C}{\mathbb C}

\newcommand{\mA}{\mathcal{A}}
\newcommand{\mB}{\mathcal{B}}
\newcommand{\mC}{\mathcal{C}}
\newcommand{\mD}{\mathcal{D}}
\newcommand{\mP}{\mathcal{P}}

\newcommand{\mL}{\mathcal{L}}

\newcommand{\mM}{\mathcal{M}}
\newcommand{\mN}{\mathcal{N}}
\newcommand{\mH}{\mathcal{H}}
\newcommand{\mK}{\mathcal{K}}
\newcommand{\mS}{\mathcal{S}}

\newcommand{\omin}{\otimes^{\min}}

\newtheorem{theorem}{Theorem}[section]
\newtheorem{lemma}[theorem]{Lemma}
\newtheorem{cor}[theorem]{Corollary}
\newtheorem{definition}[theorem]{Definition}
\newtheorem{proposition}[theorem]{Proposition}
\newtheorem{remark}[theorem]{Remark}

\begin{document}

 \title[{\tiny Distances between $C^*$-subalgebras
   and related operator algebras}]{Relations amongst the
   distances between $C^*$-subalgebras and  some canonically associated
   operator algebras}

\author[V P Gupta]{Ved Prakash Gupta} \author[S
  Kumar]{Sumit Kumar} \address{School of Physical Sciences, Jawaharlal Nehru University, New Delhi, INDIA}
\email{vedgupta@mail.jnu.ac.in} \email{sumitkumar.sk809@gmail.com}



\thanks{The second named author was  supported
  by the Council of Scientific and Industrial Research (Government
  of India) through a Senior Research Fellowship with
  CSIR File No. {\bf 09/0263(12012)/2021-EMR-I} }
\maketitle

\begin{abstract}
We prove that the Christensen distance (resp., the Kadison-Kastler
distance) between two $C^*$-subalgebras $\mA$ and $\mB$ of a
$C^*$-algebra $\mC$ is equal to that between their enveloping von
Neumann algebras $\mA^{**}$ and $\mB^{**}$ (resp., the tensor product
algebras $\mA \omin \mD$ and $\mB \omin \mD$, for any unital
commutative $C^*$-algebra $\mD$). 
\end{abstract}

\section{Introduction}
Kadison and Kastler (in \cite{KK}) introduced a notion of distance
between subspaces of $B(\mH)$ and, over the last five decades, this
notion has proved to be extremely relevant to the development of the
theory of operator algebras. They suggested that sufficiently close
operator subalgebras of $B(\mH)$ must be unitarily equivalent.  Some
pathbreaking positive results in this direction were achieved in a
series of seminal papers by Christensen. Later, Christensen
(in \cite{Ch1,Ch2}) considered another similar notion of distance and proved
some noteworthy perturbation results (\cite{Ch1, Ch2, Ch3, Ch4} and more).

In this article,  we make an  attempt to  answer the
following two fundamental questions related to these notions of distance:
\begin{enumerate}
\item For any two $C^*$-subalgebras $\mA$ and $\mB$ of a $C^*$-algebra
  $\mC$, is there any relationship amidst the (Kadison-Kastler or
  Christensen) distance between $\mA$ and $\mB$ (in $\mC$) and that
  between their enveloping von Neumann algebras $\mA^{**}$ and
  $\mB^{**}$ (in $\mC^{**}$)?

\item For any two $C^*$-algebras $\mC$ and $\mD$; and, $C^*$-subalgebras $\mA$
and $\mB$ of $\mC$, is there any relationship  amidst the
(Kadison-Kastler or Christensen) distance between $\mA$ and
$\mB$ (in $\mC$) and that between $\mA \omin \mD$ and $\mB \omin \mD$
(in $\mC \omin \mD$)?
\end{enumerate}

The first question found an indirect appearance in the  works of
Kadison-Kastler (\cite{KK}) and Christensen (\cite{Ch2}) in the
following sense:

(For the sake of clarity, as
in \cite{GK}, we denote the notions of the Kadison-Kastler distance and the
Christensen distance by $d_{KK}$ and $d_0$, respectively - see
\Cref{prelims} for their definitions.)

\begin{itemize}
  \item In \cite[Lemma 5]{KK}, Kadison and Kastler showed that, for a
    scalar $\gamma >0$ and $C^*$-algebras $\mA$ and $\mB$ acting on a
    Hilbert space $\mH$, $d_{KK}(\overline{\mA}^{SOT},
    \overline{\mB}^{SOT}) < \gamma$ whenever $d_{KK}(\mA, \mB) <
    \gamma$.

 Thus, it can be deduced readily that for any two $C^*$-subalgebras
 $\mA$ and $\mB$ of a $C^*$-algebra $\mC$, $d_{KK}({\mA}^{**},
 {\mB}^{**}) \leq d_{KK}(\mA, \mB)$ - see
 \Cref{KK-inequality}.\smallskip

\item In \cite[Theorem 6.5]{Ch2}, Christensen used the above
  observation of Kadison and Kastler to conclude that if $\mA$ and
  $\mB$ are sufficiently close $C^*$-subalgebras of a $C^*$-algebra
  $\mC$  and $\mA$ is nuclear, then $\mB$ is nuclear,
  $\mB^{**}$ and $\mA^{**}$ are isomorphic as $W^*$-algebras, and
  there exists a completely positive isometry from $\mA^*$ onto $\mB^*$.
\end{itemize}

The second question was considered to  a greater 
extent by Christensen (see, for instance, \cite{Ch2} and
\cite{Ch4}). Here are some of his interesting observations in this
direction:

\begin{theorem}\cite[Theorem 3.1]{Ch2}
Let $\mD$ be a nuclear $C^*$-algebra and,  $\mA$ and $\mB$ be 
$C^*$-subalgebras of a $C^*$-algebra $\mC$. If $\mA$
has property $D_{k}$ (for some $k \in (0,\infty)$) and $\mA
\subset_{\gamma}\mB$, then $\mA \omin \mD \subset_{6k\gamma}\mB \omin
\mD$.

In particular, when $\mA$ and $\mB$ both have property $D_{k}$, then
\[
d_0(\mA \omin \mD, \mB\omin \mD) \leq 6k\, 
d_0(\mA, \mB).
\]
\end{theorem}

Interestingly, when $\mD$ is commutative, the comparison is more
satisfying:
\begin{theorem}\cite[Theorem 3.2]{Ch2}\label{christensen-commutative}
Let $\mD$ be a commutative $C^*$-algebra,  $\mA$ and $\mB$ be 
$C^*$-subalgebras of a $C^*$-algebra $\mC$ and $\gamma$ be a positive scalar.
If $\mB \subset_{\gamma}\mA$, then $\mB \omin \mD \subset_{\gamma}\mA
\omin \mD$.

In particular, $d_0(\mA \omin \mD, \mB\omin \mD)  \leq d_0(\mA, \mB)$.  
\end{theorem}

\begin{theorem}\cite[Proposition 2.10]{Ch4}
Let $\mH$ be a Hilbert space and, $\mA$ and $ \mB$ be $C^*$-subalgebras
of $B(\mH)$.  If $\mA$ has length at most $\ell$ with length
constant at most $M$ and  $\mA \subseteq_{\gamma}\mB$ for some
$\gamma>0$, then $\mA \omin \mathbb{M}_{n}\subseteq_{\mu} \mB \omin
\mathbb{M}_{n}$ for all $n\in \mathbb{N}$, where $\mu=
M((1+\gamma)^{\ell}-1)$.

 In particular, if $\mB$ also has length at most $\ell$ with length
 constant at most $M$, then
 $$d_0(\mA \omin
\mathbb{M}_{n} , \mB \omin\mathbb{M}_{n} ) \leq M \, \left(\left(1+
d_0(\mA, \mB)\right)^\ell -1\right)$$
for all $n \in N$.
 \end{theorem}

\begin{theorem}\cite[Corollary 4.7]{Ch4}
Let $\mC$ and $\mD$ be $C^*$-algebras and $\mD$ be nuclear. Then, for
every pair $\ell, M \in \N$, there exists a constant $L_{\ell,M}$
(depending only on $\ell$ and $M$) such  that when $\mA$ and $\mB$ are
$C^*$-subalgebras of $\mC$ and $\mA$ has length at
most $\ell$ and length constant at most $M$, then
\[
d_{KK}(\mA \omin \mD, \mB\omin \mD)\leq L_{\ell,M}\, d_{KK}(\mA, \mB).
\]

\end{theorem}

Here is a quick overview of the flow of this article.

Our analysis of  the first  question yielded the following two relations:

\noindent {\em {\bf \Cref{enveloping-distance}}.  Let $\mA$ and $ \mB$
  be $C^*$-subalgebras of a $C^*$-algebra $\mC$. Then,
  \[
  d_{0}(\mA,\mB)  = d_0(\mA^{**}, \mB^{**}).\]
}

\noindent {\em {\bf \Cref{enveloping-distance-vNas}}. Let $\mM$ and
  $\mN$ be von Neumann subalgebras of a von Neumann algebra
  $\mL$. Then,
  \[
  d_{KK}(\mM,\mN)  = d_{KK}(\mM^{**}, \mN^{**}).\]
}

And, motivated by the above mentioned observations of Christensen from
\cite{Ch2, Ch4}, we prove the following results in the context of the
second question: \smallskip

 An elementary observation (\Cref{D-unital}) provides the
  following improvement of \cite[Theorem 3.2]{Ch2}: \smallskip

\noindent {\em Let $\mD$ be a commutative unital $C^*$-algebra and,
  $\mA$ and $\mB$ be $C^*$-subalgebras of a $C^*$-algebra $\mC$. Then,
  \( d_0 (\mA \omin \mD, \mB \omin \mD) = d_0 (\mA, \mB)\).} \hfill
(See \Cref{Christensen-D-commutative-unital}.) \smallskip

It turns out that the preceding equality holds with respect to the
  Kadison-Kastler distance as well:\smallskip

\noindent {\em {\bf \Cref{KK-D-commutative}.}  Let $\mD$ be a
  commutative $C^{*}$-algebra and, $\mA$ and $\mB$ be
  $C^*$-subalgebras of a $C^*$-algebra $\mC$. Then,
\[
d_{KK}(\mA\omin \mD, \mB\omin \mD)\leq d_{KK}(\mA, \mB).
\]
Moreover, if $\mD$ is  unital, then
\[
d_{KK}(\mA\omin \mD, \mB\omin \mD)= d_{KK}(\mA, \mB).
\]} 

We conclude the article with the following relations with respect to
the class of scattered $C^*$-algebras: \smallskip

\noindent {\em {\bf \Cref{scattered-relations}.}
Let $\mC$ be a $C^*$-algebra and $\mD$ be a scattered
$C^*$-algebra. Then, for any two $C^*$-subalgebras $\mA$ and $\mB$ of
$\mC$, 
\[
d_{0}(\mA,\mB)\leq d_{0}(\mA \omin \mD, \mB\omin \mD).
\]
In particular, if $\mD$ is commutative as well, then
\[
d_0(\mA, \mB)= d_{0}(\mA \omin \mD, \mB\omin \mD).
\]}

\section{Preliminaries}\label{prelims}

\subsection{Two notions of distance between subalgebras of normed algebras}
For the sake of brevity, let us first fix some notations:

 For a normed algebra $\mC$,  let
\[
\mathrm{Sub}_{\mC} :=\{ \text{subalgebras of } \mC\}; 
\]
\[
C\text{-}\mathrm{Sub}_{\mC}:=\{ \text{closed subalgebras of } \mC\};
\]
and, if $\mC$ is a $C^*$-algebra, then let
\[
C^*\text{-}\mathrm{Sub}_{\mC} :=\{ C^*\text{-subalgebras of } \mC\}.
\]

\subsubsection{Kadison-Kastler distance}
For any normed space $X$, as is standard, its closed unit ball will be
denoted by $B_1(X)$ and for any subset $S$ of $X$ and an element $x
\in X$, the distance between $x$ and $S$ is defined as
\(
d(x, S) = \inf\{ \|x - s\|: s \in S\}.
\) Also, for any $r>0$, $B_r(X):= rB_1(X)$.

Recall from \cite{KK} that the Kadison-Kastler distance between any
two subspaces $\mA$ and $\mB$ of a normed algebra $\mC$ (which we
denote by $d_{KK}(\mA,\mB)$) is defined as the Hausdorff distance
between the closed unit balls of $\mA$ and $\mB$, i.e.,
\[
d_{KK}(\mA,\mB) := \max\left\{\sup_{x \in B_1(\mA)} d(x, B_1(\mB)),
\sup_{z\in B_1(\mB)}d(z, B_1(\mA))\right\}.
\]

\begin{remark}
  Let $\mC$ be a normed algebra.  Then, the following facts are well known:
\begin{enumerate}
\item $d_{KK}(\mA,\mB) \leq 1$ for all $\mA, \mB \in \mathrm{Sub}_{\mC}$.
\item If $\mA, \mB \in C\text{-}\mathrm{Sub}_{\mC}$ and $\mA$ is a proper
  subalgebra of $\mB$, then $d_{KK}(\mA, \mB) =1$. 

  \item $d_{KK}(\mA,\mB) = d_{KK}\left(\overline{\mA}, \mB\right)=
    d_{KK}\left(\overline{\mA}, \overline{\mB}\right)$ for all
    $\mA,\mB \in \Sub_\mC$.
\item If $ \mC$ is a Banach algebra, then  $d_{KK}$ is a metric on
  $C\text{-}\mathrm{Sub}_{\mC}$.
\end{enumerate}
\end{remark}

\subsubsection{Near inclusions and Christensen distance}

 Let $\mC$ be a normed algebra. Recall from \cite{Ch1, Ch2} that, for $\mA,\, \mB
 \in \mathrm{Sub}_\mC$ and a scalar $\gamma > 0$, $\mA \subseteq_{\gamma}
 \mB$ if for each $x \in B_1(\mA)$, there exists a $
 y \in \mB$ such that $\|x-y\| \leq \gamma$; and, the Christensen
 distance between $\mA$ and $\mB$ is defined by
    \begin{equation}\label{d0}
    d_0(\mA,\mB) = \inf\{\gamma > 0 \, : \, \mA \subseteq_{\gamma}
    \mB \text{ and } \mB \subseteq_{\gamma} \mA\}.
    \end{equation}

    Further, $\mA \subset_{\gamma}
\mB$ if there exists a $\gamma_0< \gamma$ such that $\mA \subseteq_{\gamma_{0}}
\mB$.

The following is immediate and quite useful as well.
\begin{lemma}\label{d0-via-strict-near-inclusion}
  Let $\mC$ be a normed algebra. Then,
  \[
  d_0(\mA,\mB) = \inf\{\gamma > 0 \, : \, \mA \subset_{\gamma}
  \mB \text{ and } \mB \subset_{\gamma} \mA\}
  \]
  for all $\mA, \mB \in \Sub_\mC$.
  \end{lemma}

\begin{remark}
  Let $\mC$ be a normed algebra. Then, the following useful facts are
  well known:
      \begin{enumerate}
      \item $d_0(\mA, \mB) \leq 1$ for all $\mA, \mB \in \Sub_\mC$.
\item $d_0(\mA,\mB) = d_0\left(\overline{\mA}, \mB\right)= d_0\left(\overline{\mA},
  \overline{\mB}\right)$ for all $\mA,\mB \in \Sub_\mC$.
  \item $d_0$ is not a metric on
    $C\text{-}\mathrm{Sub}_\mC$ (as it  does
    not satisfy the triangle inequality). However, $d_0$ and $d_{KK}$
    are ``equivalent" in the sense that
  \[
  d_0 (\mA,\mB) \leq d_{KK}(\mA,\mB) \leq 2 d_0(\mA,\mB)
  \]
  for all $\mA,\mB \in \mathrm{Sub}_\mC$. 
\item If $\mA,\mB \in \mathrm{Sub}_\mC$ and $\mA$ is a norm closed proper
  subalgebra of $\mB$, then $d_0 (\mA,\mB ) =1$.
      \end{enumerate}
      \end{remark}

\subsection{Universal representation and enveloping von Neumann algebra of a $C^*$-algebra}\( \)

For any $C^*$-algebra $\mC$, let $\mS(\mC)$ denote its state space.
For each $\varphi \in \mS(\mC)$, let $\pi_{\varphi}: \mC \to B(\mH_\varphi)$
denote the GNS representation of $\mC$ associated with $\varphi$; and,
let $\mH_\mC:= \oplus_{\varphi \in \mS(\mC)} \mH_\varphi$. Recall that
the representation $\pi_{\mC} := \oplus_{\varphi \in \mS(\mC)}
\pi_{\varphi}: \mC \to B(\mH_\mC)$ is faithful and is called the
universal representation of $\mC$. Further, $\pi_\mC$ extends to a
surjective linear isometry $\tilde{\pi}_\mC : \mC^{**} \to
\pi_\mC(\mC)''$, which is also a $(w^*,
\sigma\text{-weak})$-homeomorphism; and, via this identification,
$\mC^{**}$ inherits a $W^*$-algebra structure, which is known as the
enveloping von Neumann algebra of $\mC$.

  We shall need the following well-known facts - see, for
  instance, \cite[$\S 3.7$ ]{Ped}.

\begin{proposition}\label{continuous=WOT-continuous}
Let $\mC$ be a $C^*$-algebra. Then, every  continuous linear functional
on $\pi_{\mC}(\mC)$ is W.O.T.-continuous.
\end{proposition}

  \begin{proposition}\label{universal-restriction}
  Let $\mC$ be a $C^*$-algebra and $\mB$ be a $C^*$-subalgebra of
  $\mC$. Then, the following hold:
  \begin{enumerate}
  \item If $i: \mB \to \mC$ denotes the natural inclusion, then
    $i^{**}$ is an isometric $*$-isomorphism from $\mB^{**}$ onto
    $\overline{J_{\mC}(i(\mB))}^{w^*},$ where $J_\mC$ is the canonical
    linear isometry from $\mC$ into $\mC^{**}$.
    \item There exists an isometric $*$-isomorphism from the enveloping von
      Neumann algebra $\mB^{**}$ (of $\mB$) onto 
      $\overline{\pi_{\mC}(\mB)}^{W.O.T.} $, which maps $i(b)$ to
      $\pi_\mC(b)$ for all $b \in \mB$. Moreover, the $*$-isomorphism is
      also a $(w^*, \sigma\text{-weak})$-homeomorphism.
  \end{enumerate}
\end{proposition}

\subsection{Minimal tensor product}

  Let $\mC$ and $\mD$ be $C^*$-algebras and  $\mC \otimes \mD$ denote
  their algebraic tensor product. Then,
  for any $z \in  \mC \otimes \mD$, its  $C^*$-minimal (tensor) norm is
  given by
\begin{eqnarray*}
\lefteqn{  \|z\|_{\min} =}\\ & & \sup \left\{
  \frac{(\varphi \otimes \psi) \Big( ( w^*z^*zw)^{\frac{1}{2}}
    \Big)}{(\varphi\otimes \psi)\Big( (w^*w)^{\frac{1}{2}}\Big)}:
  \varphi \in \mS(\mC), \psi \in \mS(\mD),  w \in \mC \otimes \mD, (\varphi \otimes \psi)(w) \neq 0
  \right\}.
\end{eqnarray*}
The completion of
$\mC \otimes \mD$ with respect to $\|\cdot\|_{\min}$ is a
$C^*$-algebra and is denoted by $\mC \omin \mD$.

\begin{remark}\label{min-facts}
  Let  $\mC$ and $\mD$ be $C^*$-algebras. The following well-known facts will be used ahead:
  \begin{enumerate}
    \item If $\pi: \mC \to B(\mH)$ and $\sigma : \mD \to B(\mK)$ are
      two faithful representations, then
  \[
    \Big\|\sum_i c_i \otimes d_i\Big\|_{\min} = \Big\|\sum_i \pi(c_i) \otimes
    \sigma(d_i)\Big\|
  \]
for all $\sum_i c_i \otimes d_i \in \mC \otimes \mD$.
\item The minimal tensor product is an injective tensor product, i.e.,
  for any $\mA \in C^*\text{-}\Sub_\mC$ and $\mB \in
  C^*\text{-}\Sub_\mD$, there exists a unique isometric
  $*$-isomorphism from $\mA \omin \mB$ onto the $C^*$-subalgebra
  $\overline{\mA \otimes \mB}^{\|\cdot\|_{\min}}$ of $ \mC \omin \mD$
  which fixes $\mA \otimes \mB$.  (This allows us to consider $\mA
  \omin \mB$ as a $C^*$-subalgebra of $\mC \omin \mD$.)
  \end{enumerate}
  \end{remark}

\section{Relations amidst distances between $C^*$-subalgebras and their enveloping von Neumann algebras}

The proof of the following theorem uses  some techniques employed in
the proofs of \cite[Theorem 3.1]{Ch2} and \cite[Lemma 5]{KK}.

\begin{theorem}\label{enveloping-distance}
  Let $\mA$ and $ \mB$ be $C^*$-subalgebras of a $C^*$-algebra
  $\mC$. Then,
  \[
d_{0}(\mA,\mB) =  d_0(\mA^{**}, \mB^{**}) .
\]
\end{theorem}
\begin{proof}
  Note that, by \Cref{universal-restriction}(1), we can identify
  $\mB^{**}$ and $\mA^{**}$ with the $w^*$-closed $*$-subalgebras
  $\overline{J_\mC(\mB)}^{w^*}$ and $\overline{J_\mC(\mA)}^{w^*}$ of
  $\mC^{**}$, respectively. Thus, via these identifications, in view
  of \Cref{universal-restriction}(2), the isometric $*$-isomorphism
  $\tilde{\pi}_{\mC}: \mC^{**} \to \pi_{\mC}(\mC)''$ maps ${\mA^{**}}$
  (resp., $\mB^{**}$) isometrically isomorphically onto
  $\overline{\pi_{\mC}(\mA)}^{W.O.T.}$ (resp.,
  $\overline{\pi_{\mC}(\mB)}^{W.O.T.}$).  So, it suffices to
  show that
  \[
  d_0\left(\pi_\mC(\mA), \pi_\mC(\mB)\right) =
  d_0\left(\overline{\pi_{\mC}(\mA)}^{W.O.T.},
  \overline{\pi_{\mC}(\mB)}^{W.O.T.}\right),
  \]
  which clearly holds if we can show, for any scalar $\gamma > 0$,
  that $\overline{\pi_{\mC}(\mA)}^{W.O.T.}\subset_{\gamma}
  \overline{\pi_{\mC}(\mB)}^{W.O.T.}$ if and only if
  $\pi_{\mC}(\mA)\subset_{\gamma} \pi_{\mC}(\mB)$.\smallskip

First, suppose that $\overline{\pi_{\mC}(\mA)}^{W.O.T.}
\subset_{\gamma}\overline{\pi_{\mC}(\mB)}^{W.O.T.}$.

 By definition, there exists a $\gamma_{0}<\gamma$ such that
 $\overline{\pi_{\mC}(\mA)}^{W.O.T.}
 \subseteq_{\gamma_{0}}\overline{\pi_{\mC}(\mB)}^{W.O.T.}$. Fix a
 scalar $\gamma_1$ satisfying $ \gamma_0 < \gamma_1 < \gamma$.  We
 show that $\pi_{\mC}(\mA)\subseteq_{\gamma_{1}}\pi_{\mC}(\mB)$. This
 will then imply that
 $\pi_{\mC}(\mA)\subset_{\gamma}\pi_{\mC}(\mB)$.\smallskip

 Let $z \in B_{1}(\pi_{\mC}(\mA))$. Without loss of generality, we can
 assume that $z \in B_{1}(\pi_{\mC}(\mA))\setminus
 \pi_{\mC}(\mB)$. Then, by the Hahn-Banach theorem, there exists a
 $\varphi\in \pi_{\mC}(\mC)^{*}$ such that
 $\varphi(\pi_{\mC}(\mB))=(0)$, $\varphi(z)= d(z, \pi_{\mC}(\mB))$ and
 $\|\varphi\|=1$. In view of \Cref{continuous=WOT-continuous},
 $\varphi $ is W.O.T.-continuous and hence $\sigma$-weakly continuous.
 Thus, by the Hahn-Banach theorem (for the $\sigma$-weak topology), we
 can extend $\varphi$ uniquely to a $\sigma$-weakly continuous linear
 functional $\tilde{\varphi}: \pi_{\mC}(\mC)^{''}\rightarrow
 \mathbb{C}$. Clearly,
 ${\tilde{\varphi}}_{\restriction_{\overline{\pi_{\mC}(\mB)}^{W.O.T.}}}=0$
 and  an appeal to the Kaplansky's density
 theorem shows that $\|\tilde{\varphi}\|=1$, as well.  We assert that \,
 $\|{\tilde{\varphi}}_{\restriction_{\overline{\pi_{\mC}(\mA)}^{W.O.T.}}}\|\leq
 \gamma_{0}$.

  Let $x\in B_{1}(\overline{\pi_{\mC}(\mA)}^{W.O.T.})$. Since
  $\overline{\pi_{\mC}(\mA)}^{W.O.T.}  \subseteq_{\gamma_0}
  \overline{\pi_{\mC}(\mB)}^{W.O.T.}$, there exists a $y\in
  \overline{\pi_{\mC}(\mB)}^{W.O.T.}$ such that $\|x-y\|\leq
  \gamma_{0}$. Thus,
\[
|\tilde{\varphi}(x)|= |\tilde{\varphi}(x-y)|\leq \|x-y\|\leq \gamma_{0};
\]
so that, $\|\tilde{\varphi}_{\restriction_{\overline{\pi_{\mC}(\mA)}^{W.O.T.}}}\|\leq
\gamma_{0}$; as was asserted.

Thus, \( d(z, \pi_{\mC}(\mB)) = |\varphi(z)|= |\tilde{\varphi
}(z)|\leq \gamma_{0} <\gamma_{1}.  \) So, there exists a $w \in
\pi_{\mC}(\mB)$ such that $\|z - w\|\leq \gamma_{1}<\gamma$.  Thus,
$\pi_{\mC}(\mA)\subseteq_{\gamma_1}\pi_{\mC}(\mB)$ and, hence,
$\pi_{\mC}(\mA)\subset_{\gamma}\pi_{\mC}(\mB)$.\smallskip
 
 Conversely, suppose that  $\pi_{\mC}(\mA) \subset_{\gamma}\pi_{\mC}(\mB)$.

 Fix a $\beta< \gamma$ such that
 $\pi_{\mC}(\mA)\subset_{\beta}\pi_{\mC}(\mB)$. We shall show that
 $\overline{\pi_{\mC}(\mA)}^{W.O.T.}
 \subseteq_{\beta}\overline{\pi_{\mC}(\mB)}^{W.O.T.}$, as well.

 Let $x\in B_{1}(\overline{\pi_{\mC}(\mA)}^{W.O.T.}) $. It suffices to
 show that there exists a $z \in \overline{\pi_\mC(\mB)}^{W.O.T.}$
 such that $|\langle (x - z) \zeta, \eta \rangle| \leq \beta$ for all
 $\zeta, \eta \in \mH$. Towards this end, let $\delta = 1+ \gamma$ and,
 for every ordered pair $(\zeta, \eta)\in B_1(\mcal{H})\times
 B_1(\mH)$, consider
\[
S_{\zeta,\eta}=\left\{y\in
B_{\delta}\left(\overline{\pi_{\mC}(\mB)}^{W.O.T.}\right) : \,
|\left\langle(x-y)\zeta, \eta \right\rangle| \leq \beta \right\}.
\]

\textbf{Claim (1):} $S_{\zeta,\eta}$ is non-empty and W.O.T. closed
in $B_{\delta}\left(\overline{\pi_{\mC}(\mB)}^{W.O.T.}\right)$.\smallskip

Clearly, $S_{\zeta,\eta}$ is W.O.T.-closed in
$B_{\delta}\left(\overline{\pi_{\mC}(\mB)}^{W.O.T.}\right)$. We just need to show
that it is non-empty. By the Kaplansky's density theorem, there exists
a net $\{x_{\alpha}\}\subset B_{1}(\pi_{\mC}(\mA))$ such that
$x_{\alpha}\rightarrow x$ in the weak operator topology.  Since
$\pi_{\mC}(\mA)\subset_{\beta}\pi_{\mC}(\mB)$, for each $\alpha$,
there exists a $y_{\alpha}\in \pi_{\mC}(\mB)$ such that
$\|x_{\alpha}-y_{\alpha}\|< \beta$. Clearly, $\|y_{\alpha}\|<\delta$
for all $\alpha$.

  Note that 
  $B_{\delta }\left(\overline{\pi_{\mC}(\mB)}^{W.O.T.}\right)$ is
  W.O.T.-compact, by \Cref{universal-restriction}(2); so, there exists a
  W.O.T.-convergent subnet $\{y_{\alpha_{j}}\}$ of $\{y_{\alpha}\}$. Let $y =
 \lim_{j} y_{\alpha_{j}}$ in  the weak operator topology. Since
  $|\left\langle(x_{\alpha}-y_{\alpha})\zeta, \eta \right\rangle| \leq
  \beta$ for every $\alpha$, it follows that 
\[
|\left\langle(x-y)\zeta, \eta \right\rangle| =
\lim_{j}|\left\langle(x_{\alpha_{j}}-y_{\alpha_{j}})\zeta, \eta
\right\rangle| \leq \beta,
\]
i.e., $y\in S_{\zeta,\eta}$. Thus, $S_{\zeta,\eta}\neq \emptyset.$

\medskip

\textbf{Claim (2):} The collection $\{S_{\zeta,\eta} : (\zeta, \eta)\in
  B_{1}(\mcal{H})\times B_{1}(\mcal{H})\}$ satisfies the finite
intersection property in
$B_{\delta}\left(\overline{\pi_{\mC}(\mB)}^{W.O.T.}\right)$.\smallskip

Let $\{(\zeta_{1}\times \eta_{1}),...,(\zeta_{n}\times \eta_{n})\} \subset
B_{1}(\mcal{H})\times B_{1}(\mcal{H})$. Fix a  $0 < \lambda < \beta$ such
that $\pi_{\mC}(\mA)\subset_{\lambda}\pi_{\mC}(\mB)$.

Since $x \in B_1\left(\overline{\pi_{\mC}(\mA)}^{W.O.T.}\right)$, by
the Kaplansky's density theorem again, there exists a $z \in
B_{1}(\pi_{\mC}(\mA))$ such that $|\left\langle(z - x)\zeta_{i},
\eta_{i} \right\rangle|< \beta-\lambda$ for all $ 1 \leq i\leq
n$. Further, by the choice of $\lambda$, there exists a $ w \in
\pi_{\mC}(\mB)$ such that $\|z - w\|< \lambda$. This implies that
$|\left\langle(z - w)\zeta_{i}, \eta_{i} \right\rangle|< \lambda$ for
all $1 \leq i\leq n$. Thus,
\[
|\left\langle(x - w )\zeta_{i}, \eta_{i} \right\rangle|\leq
|\left\langle(x-z)\zeta_{i}, \eta_{i} \right\rangle|+
|\left\langle(z-w)\zeta_{i}, \eta_{i} \right\rangle|< \beta\,\,\text{
  for all }\,\, 1\leq i\leq n,
\]
i.e.,  $w \in \cap_{i=1}^{n}S_{\zeta_{i},\eta_{i}}$. This
proves Claim (2). \medskip

Thus,  $B_{\delta}\left(\overline{\pi_{\mC}(\mB)}^{W.O.T.}\right)$ being
W.O.T.-compact, there is a $z\in \cap \{ S_{\zeta,\eta} :
(\zeta,\eta)\in B_{1}(\mcal{H})\times B_{1}(\mcal{H})\}$. In
particular, $z \in \overline{\pi_\mC(\mB)}^{W.O.T.}$ and
$|\left\langle(x-z)\zeta, \eta \right\rangle|\leq \beta< \gamma$ for
all $(\zeta,\eta)\in B_{1}(\mcal{H})\times B_{1}(\mcal{H})$, as was desired.
\end{proof}
At this moment, it is not clear to us whether the preceding equality
holds with respect to the Kadison-Kastler distance or not. However, in
view of \Cref{universal-restriction}, the following comparison can be
deduced easily from \cite[Lemma 5]{KK}:

\begin{proposition}\cite[Lemma 5]{KK}\label{KK-inequality}
  Let $\mA$ and $\mB$ be $C^*$-subalgebras of a $C^*$-algebra $\mC$. Then,
  \[
d_{KK}(\mA^{**}, \mB^{**}) \leq d_{KK}(\mA,\mB).
  \]
  \end{proposition}
When we restrict to the category of von Neumann algebras, we observe
that even the Kadison-Kastler distance between von Neumann subalgebras
and their biduals are equal. In order to see this, we need the
following well-known result - see, for instance,
\cite[III.5.2.15]{BL}.

\begin{proposition}\label{vNa-enveloping-ce}
Let $\mL$ be a von Neumann algebra. Then, there exists a surjective
normal $*$-homomorphism $E_\mL: \mL^{**} \to \mL$ that fixes $\mL$
(where $\mL$ is identified with its canonical isometric embedding in
$\mL^{**}$).

In particular, for any von Neumann subalgebra $\mM$ of $\mL$, the
restriction map $E_\mM:={E_\mL}_{\restriction_{\mM^{**}}}$  maps $\mM^{**}$ onto $\mM$.
\end{proposition}

\begin{theorem}\label{enveloping-distance-vNas}
Let $\mM$ and $\mN$ be von Neumann subalgebras of a von Neumann
algebra $\mL$. Then,
  \[
d_{KK}(\mM,\mN) =  d_{KK}(\mM^{**}, \mN^{**}) .
\]
\end{theorem}
\begin{proof}
In view of \Cref{KK-inequality}, it just remains to show that
$d_{KK}(\mM,\mN)\leq d_{KK}(\mM^{**},\mN^{**})$.

Let $\epsilon>0$ and $x\in B_{1}(\mM)\subset B_{1}(\mM^{**})$. Then,
there exists a $y\in B_{1}(\mN^{**})$ such that \( \|x-y\|<
d_{KK}(\mM^{**},\mN^{**})+ \epsilon$, since $d(x, B_{1}(\mN^{**}))\leq
d_{KK}(\mM^{**},\mN^{**})$.

Consider the surjective normal $*$-homomorphisms $E_\mL: \mL^{**} \to
\mL$ and its restriction $E_\mN: \mN^{**} \to \mN$ as in
\Cref{vNa-enveloping-ce}. Let $z= E_{\mN}(y)\in B_{1}(\mN)$. Since
$E_\mL(x) =x $ and $E_\mL(y) = E_\mN(y) = z$, we observe that
\[
\|x-z\|=\|E_{\mL}(x-y)\|\leq \|x-y\|< d_{KK}(\mM^{**},\mN^{**})+ \epsilon;
\]
so that, 
\[
d(x,B_{1}(\mN))< d_{KK}(\mM^{**},\mN^{**})+ \epsilon. 
\]
Thus, 
\[
\sup_{x\in B_{1}(\mM)}d(x,B_{1}(\mN))\leq d_{KK}(\mM^{**},\mN^{**})+ \epsilon.
\]
By symmetry, 
\[
\sup_{y\in B_{1}(\mN)}d(y,B_{1}(\mM))\leq d_{KK}(\mM^{**},\mN^{**})+ \epsilon
\]
as well. Hence, 
\[
d_{KK}(\mM,\mN)\leq d_{KK}(\mM^{**},\mN^{**})+ \epsilon.
\]
Since $\epsilon>0$ was arbitrary, we get
\[
d_{KK}(\mM,\mN)\leq d_{KK}(\mM^{**},\mN^{**})
\]
and we are done. \end{proof}

\section{Kadison-Kastler and Christensen distance between tensor product subalgebras}

\begin{proposition}\label{D-ce-P}
Let $\mC$ and $\mD$ be $C^*$-algebras; $\mA, \mB\in
C^*\text{-}\Sub_\mC$ and $\mP\in C^*\text{-}\Sub_\mD$. If there exists
a conditional expectation from $\mD$ onto $\mP$, then
\[
d_{0}(\mA \omin \mP, \mB\omin \mP)\leq d_{0}(\mA \omin \mD, \mB\omin \mD)
\]
and 
\[
d_{KK}(\mA \omin \mP, \mB\omin \mP)\leq d_{KK}(\mA \omin \mD, \mB\omin \mD).
\]
\end{proposition}
\begin{proof}
Let $E: \mD \to \mP$ be a conditional expectation. Let $\epsilon>0$
and fix a $\gamma_{0}>0$ such that
\[
d_{0}(\mA \omin \mD, \mB\omin \mD)<  \gamma_{0} < d_{0}(\mA \omin
\mD, \mB\omin \mD) + \epsilon.
\]
This implies that $\mA \omin \mD\subseteq_{\gamma_{0}}\mB\omin \mD$
and $\mB \omin \mD\subseteq_{\gamma_{0}}\mA\omin \mD$.

Let $x\in B_{1}(\mA \omin \mP)\subseteq B_{1}(\mA \omin \mD)$. Then,
there exists a $y\in \mB\omin \mD$ such that $\|x-y\|_{\min}\leq
\gamma_{0}$. Consider the conditional expectation
$\mathrm{id}_{\mC}\omin E: \mC \omin \mD \to \mC \omin \mP$. Clearly,
it maps $\mB\omin \mD$ onto $\mB\omin \mP$. Thus, $y_{0}:=
(\mathrm{id}_{\mC}\omin E)(y)\in \mB\omin \mP$ and
\begin{eqnarray*}
\|x-y_{0}\|_{\min} &=& \|(\mathrm{id}_{\mC}\omin E)(x-y)\|_{\min}\\
&\leq & \|x-y\|_{\min}\leq \gamma_{0}.
\end{eqnarray*}
So, $\mA \omin \mP\subseteq_{\gamma_{0}}\mB\omin
\mP$. Similarly,  $\mB \omin \mP\subseteq_{\gamma_{0}}\mA\omin
\mP$; so that 
\[
d_{0}(\mA \omin \mP, \mB\omin \mP)\leq \gamma_{0}< d_{0}(\mA \omin
\mD, \mB\omin \mD)+ \epsilon.
\]
Hence,
\[
d_{0}(\mA \omin \mP, \mB\omin \mP) \leq d_{0}(\mA \omin \mD, \mB\omin \mD).
\] 

Since $\mathrm{id}_\mC \omin E$ is a contraction, the proof for the
Kadison-Kastler distance follows verbatim and we leave the details to
the reader.
\end{proof}

Since every state on a unital $C^*$-algebra is a conditional
expectation onto $\C$, we immediately deduce the following:

\begin{cor}\label{KK-n-d0-inequality}\label{D-unital}
  Let $\mC $ and $\mD $ be $C^*$-algebras. If $\mD$ is unital, then
\[
d_{KK}(\mA, \mB)\leq d_{KK}(\mA \omin \mD, \mB\omin \mD)
\]
 and \[
d_{0}(\mA, \mB)\leq d_{0}(\mA \omin \mD, \mB\omin \mD)
\]
for all $\mA\,,\mB \in C^*\text{-}\Sub_\mC$.
\end{cor}

The following well-known fact about the so-called `left slice maps'
will be useful ahead:
\begin{lemma}\label{slice-map}
  Let $\mC$ and $ \mD$ be two $C^*$-algebras and $\varphi \in
  \mD^*$. Then, there exists a unique (left slice map) $L_{\varphi}
  \in B(\mC \omin \mD, \mC)$ such that $\|L_\varphi\|=\|\varphi\|$ and
  \( L_{\varphi} (c \otimes d) = c \varphi(d) \) for all $c \in \mC$
  and $ d \in \mD$.
\end{lemma}

It is not yet clear (at least to us) whether we can drop unitality of
$\mD$ in \Cref{D-unital}. However, if we restrict to
finite-dimensional $*$-subalgebras of $\mC$,  we have the
following positive answer:

\begin{proposition}\label{Finite-dim}
Let $\mC$ and $\mD$ be $C^*$-algebras. Then, for any two
finite-dimensional $*$-subalgebras $\mA$ and $\mB$ of a $C^*$-algebra
$\mC$,
\[
d_{KK}(\mA, \mB)\leq d_{KK}(\mA \omin \mD, \mB\omin \mD)
\]
and
\[
d_0(\mA, \mB)\leq d_{0}(\mA \omin \mD, \mB\omin \mD).
\]
\end{proposition}

\begin{proof}
Fix a $\varphi \in S(\mD)$ and an approximate unit $\{u_{\lambda}\}$
for $\mD$ in $B_{1}(\mD).$ It is a well-known fact that
$\lim_{\lambda}\varphi(u_{\lambda}) = \|\varphi\| = 1$ - see
\cite[Proposition 3.1.4]{Ped}.

Let $a\in B_{1}(\mA)$ and $\epsilon>0$. Then, $a\otimes u_{\lambda}\in B_{1}(\mA
\omin \mD)$ for all $\lambda$. Thus, for each $\lambda$, there exists
a $w_{\lambda}\in B_{1}(\mB \omin \mD)$ such that
\[
\|a\otimes u_{\lambda}- w_{\lambda}\|_{\min}\leq d_{KK}(\mA \omin \mD,
\mB\omin \mD)+ \epsilon.
\]
Consider the left slice map $L_\varphi \in B(\mC \omin \mD, \mC)$ as
in \Cref{slice-map}. Note that $\{L_{\varphi}(w_{\lambda})\}$ is a
bounded net in the finite-dimensional space $\mB$; so, it has a
convergent subnet in $\mB$, say, $\{ L_{\varphi}(w_{\lambda_{i}}) : i
\in I\}$, with limit $b_{0}\in B_{1}(\mB)$. Then,
\begin{eqnarray*}
\|a-b_{0}\| &=& \|\lim_{i }\varphi(u_{\lambda_{i}})a- \lim_{i
}L_{\varphi}(w_{\lambda_{i}})\| \\ &=& \lim_{i}\|L_{\varphi}(a\otimes
u_{\lambda_{i}}- w_{\lambda_{i}})\| \\ &\leq & d_{KK}(\mA \omin \mD,
\mB\omin \mD)+ \epsilon. \qquad (\text{since $L_{\varphi}$ is a
  contraction})\\
\end{eqnarray*}  
Thus, 
$$\sup_{a\in B_{1}(\mA)} d(a, B_{1}(\mB))\leq d_{KK}(\mA \omin \mD,
\mB\omin \mD)+ \epsilon.
$$
Likewise,  
\[
\sup_{b\in B_{1}(\mB)} d(b, B_{1}(\mA))\leq d_{KK}(\mA \omin \mD,
\mB\omin \mD)+ \epsilon.
\]
Hence, $d_{KK}(\mA, \mB)\leq d_{KK}(\mA \omin \mD, \mB\omin \mD)+
\epsilon$. Since $\epsilon>0$ was arbitrary, we get
\[
d_{KK}(\mA, \mB)\leq d_{KK}(\mA \omin \mD, \mB\omin \mD).
\]

Similarly, one can do for the Christensen distance as well.
\end{proof}

\color{black}
\subsection{Tensor product with commutative $C^*$-algebras}\( \)

In view of \Cref{D-unital}, we readily deduce the following improvement of
\cite[Theorem 3.2]{Ch2} when the commutative $C^*$-algebra is unital
(see \Cref{christensen-commutative} above for the statement).

\begin{theorem}\label{Christensen-D-commutative-unital}
Let $\mD$ be a commutative unital $C^{*}$-algebra and, $\mA$ and $\mB$
be $C^*$-subalgebras of a $C^*$-algebra $\mC$. Then,
\[
d_{0}(\mA\omin \mD, \mB\omin \mD)= d_{0}(\mA, \mB).
\]
\end{theorem}

The proof of the following theorem is an appropriate adaptation of
that of \cite[Theorem 3.2]{Ch2}:

\begin{theorem}\label{KK-D-commutative}
 Let $\mD$ be a commutative $C^{*}$-algebra and, $\mA$ and $\mB$ be
 $C^*$-subalgebras of a $C^*$-algebra $\mC$. Then,
\[
d_{KK}(\mA\omin \mD, \mB\omin \mD)\leq d_{KK}(\mA, \mB).
\]
Moreover, if $\mD$ is  unital, then
\[
d_{KK}(\mA\omin \mD, \mB\omin \mD)= d_{KK}(\mA, \mB).
\] 
\end{theorem}

\begin{proof}
  Suppose that $\mD= C_{0}(\Omega)$ for some locally compact Hausdorff
  space $\Omega$. Then, it is a standard fact that there exists a
  $*$-isomorphism from $\mC \omin \mD$ onto $ C_{0}(\Omega, \mC)$,
  which maps $\mB \omin \mD$ onto $ C_{0}(\Omega, \mB)$ and $\mA \omin
  \mD $ onto $ C_{0}(\Omega, \mA)$. Thus, it suffices to show that
  $d_{KK}(C_0(\Omega, \mA), C_0(\Omega, \mB) \leq d_{KK}(\mA, \mB)$.

Let $\delta >0$ and $f\in B_{1}(C_{0}(\Omega, \mB))$. Then, there
exists a compact set $K$ in $\Omega$ such that $\|f(t)\|\leq \delta$
for all $t\in \Omega\setminus K$. Further, by the compactness of $K$,
there exists a finite collection $\{b_1,b_2,..., b_n\}\subset
B_{1}(\mB)\cap f(K) $ such that $K \subseteq \cup_{i=1}^n V_{i}$,
where $V_{i}:= f^{-1}( B^o_{\frac{\delta}{2}}(b_{i}))$ and
$B^o_{\frac{\delta}{2}}(b_i)$ denotes the open ball of radius
$\frac{\delta}{2}$ with center $b_i$ in $\mB$. Note that, for any two
points $t$ and $s$ in any $V_i$, $\|f(t) -f(s) \| < \delta$.

Let $\{h_i\}_{i=1}^n \subset C_{c}(\Omega)$ be a partition of unity on
$\Omega$ subordinate to the open cover $\{V_{i}\}_{i=1}^n$, i.e.,
$\mathrm{supp}(h_i)\subset V_{i}$, $\, 0\leq h_{i}\leq \mathbf{1}$ for
every $ i$, $ \sum_{i=1}^{n}h_{i}= \mathbf{1}$ on $K$ and
$\sum_{i=1}^{n}h_{i}\leq 1$ on $\Omega$.
  
 For each $1 \leq j\leq n$, fix a $t_{j}\in \mathrm{supp}(h_j)$. So,
 $f(t_{j})\in B_{1}(\mB)$ and, since $d_{KK}(\mA, \mB)< d_{KK}(\mA,
 \mB) + \delta$, there exists an $a_{j}\in B_{1}(\mA)$ such that
 $\|f(t_{j})-a_{j}\| \leq d_{KK}(\mA, \mB) + \delta$.  Consider $g:=
 \sum_{i=1}^{n}h_{i}a_{i} \in B_{1}(C_{0}(\Omega, \mA))$.
 
We show that $\|f-g\|\leq d_{KK}(\mA, \mB) + 4\delta$. Let $t\in \Omega$.

\noindent \textbf{Case (1).}\, If $t\in K$, then
\begin{eqnarray*}
\|f(t)-g(t)\| &\leq&  \|f(t)-\sum_{i=1}^{n}h_{i}(t)a_{i}  \|\\ &=&
 \|\sum_{i=1}^{n}h_{i}(t)f(t)-\sum_{i=1}^{n}h_{i}(t)a_{i}  \| \\ &=&
\|\sum_{i=1}^{n}h_{i}(t)(f(t)-a_{i})\|\\ &\leq &
\sum_{i=1}^{n}h_{i}(t)\|(f(t)-a_{i})\|\\ &\leq &
\sum_{i=1}^{n}h_{i}(t)\|(f(t)-f(t_{i}))\|+
\sum_{i=1}^{n}h_{i}(t)\|(f(t_{i})-a_{i})\|\\ &\leq & \delta +
d_{KK}(\mA, \mB) + \delta ,
\end{eqnarray*} 
where the  last inequality holds because when $t \in V_j$ for
some $j$, then $\|f(t) - f(t_j)\|< \delta$ and when $t \notin V_r$ for
some $r$, then $h_r(t) = 0$.\smallskip

\textbf{Case (2).}\, If $t\in \Omega \setminus \cup_{i=1}^n V_{i}$,
then $g(t)=0$ and $\|f(t)\|\leq \delta < d_{KK}(\mA, \mB)+ 4\delta
$.  \smallskip

\textbf{Case (3).}\, If $t\in \cup_{i=1}^n V_{i}\setminus K$. Then, as
$\sum_{i=1}^{n}h_{i}\leq 1$,  we observe that 
\begin{eqnarray*}
\lefteqn{\|f(t)-g(t)\|}\\ & = & \|f(t)-\sum_{i=1}^{n}h_{i}(t)a_{i}\|\\
& \leq & \|f(t)\| + \|\sum_{i=1}^{n}h_{i}(t)a_{i}\| \\
& \leq &  \delta + \|\sum_{i=1}^{n}h_{i}(t)(a_{i}- f(t_{i})+ f(t_{i}))\|\\
&\leq & \delta + \|\sum_{i=1}^{n}h_{i}(t)(a_{i} - f(t_{i}))\| + \|\sum_{i=1}^{n}h_{i}(t)f(t_{i})\|\\
&\leq &  \delta + \sum_{i=1}^{n}h_{i}(t)\|a_{i} - f(t_{i})\|+ \| \sum_{i=1}^{n}h_{i}(t)(f(t_{i}) - f(t) + f(t))\|\\
&\leq & \delta + (d_{KK}(\mA, \mB) + \delta) \sum_{i=1}^{n}h_{i}(t) + \| \sum_{i=1}^{n}h_{i}(t)(f(t_{i}) - f(t))\| + \|\sum_{i=1}^{n}h_{i}(t)f(t)\|  \\
& \leq &  \delta +d_{KK}(\mA, \mB) + \delta +  \sum_{i=1}^{n}h_{i}(t)\|f(t_{i}) - f(t)\| + \sum_{i=1}^{n}h_{i}(t)\|f(t)\| \\
& \leq & \delta + d_{KK}(\mA, \mB) + \delta + \delta + \delta \\
& \leq &  d_{KK}(\mA, \mB) + 4\delta.
\end{eqnarray*} 
Thus, $\|f-g\|\leq d_{KK}(\mA, \mB) + 4\delta$. 

In particular, 
\[
d(f, B_{1}(C_{0}(\Omega, \mA))) \leq d_{KK}(\mA, \mB) + 4\delta .
\]
Thus, 
\[
\sup_{f\in B_{1}(C_{0}(\Omega, \mB))}d(f, B_{1}(C_{0}(\Omega, \mA)))
\leq d_{KK}(\mA, \mB) + 4\delta.
\]
By symmetry, we also have, 
\[
\sup_{g\in B_{1}(C_{0}(\Omega, \mA))}d(g, B_{1}(C_{0}(\Omega, \mB)))
\leq d_{KK}(\mA, \mB) + 4\delta.
\] 
Hence, $d_{KK}(C_{0}(\Omega, \mA), C_{0}(\Omega, \mB))\leq d_{KK}(\mA,
\mB) + 4\delta$. Since $\delta>0 $ was arbitrary, we get
\[
d_{KK}(\mA\omin \mD, \mB\omin \mD)\leq d_{KK}(\mA, \mB).
\]
\end{proof}

By \Cref{Finite-dim}, \Cref{KK-D-commutative} and
\Cref{christensen-commutative} we obtain the following:
\begin{cor}
Let $\mD$ be a commutative  $C^*$-algebra. Then, for any two finite-dimensional $*$-subalgebras $\mA$ and $\mB$ of $\mC$, we have
\[
d_{KK}(\mA, \mB)= d_{KK}(\mA \omin \mD, \mB\omin \mD)
\]
and
\[
d_0(\mA, \mB)= d_{0}(\mA \omin \mD, \mB\omin \mD).
\]
\end{cor}
\subsection{Tensor product with scattered $C^*$-algebras}\( \)

First, we briefly recall the notion of support of a positive linear
functional on a $C^*$-algebra.

If $\omega$ is a positive linear functional on a $C^*$-algebra $\mA$,
then identifying $\mA$ with $\pi_\mA(\mA)$, we can extend it to a
normal positive linear functional on $\mA^{**}$ (by
\Cref{continuous=WOT-continuous} and the Hahn-Banach theorem for the
weak$^*$-topology on $\mA^{**}$), say $\hat{\omega}$. The support
projection of $\hat{\omega}$ in $\mA^{**}$ is called the support of
$\omega$ - see, for instance, \cite[Page no.-140]{Tak}.

\begin{definition}\cite{J}
  Let $f$ be a positive linear functional on a $C^*$-algebra $\mC$
  with support $s$. Then, $f$ is called atomic, if for
  each projection $p\in \mC^{**}$ with $p\leq s$, there
  exists a minimal projection $q\in \mC^{**}$ such that
  $q\leq p$.
\end{definition}

\begin{definition} \cite{J} A $C^*$-algebra $\mC$ is called
scattered, if each positive linear functional on $\mC$ is atomic.
\end{definition}

Here are some standard examples of scattered $C^*$-algebras. 
\begin{enumerate}
\item Every finite dimensional $C^*$-algebra is scattered.
\item The $C^*$-algebra $\mathcal{K}$ of compact operators on
  $l^2(\mathbb{N})$ is scattered.
\item If $X$ is a scattered locally compact Hausdorff space, then
  $C_0(X)$ is scattered.  (Recall that by a scattered topological
  space we mean that it does not contain any perfect subset.)
\end{enumerate}

The following theorem is due to  Huruya (\cite{T}), which was also
proved later, by a different method, by Quigg (\cite{Q}):
\begin{theorem}\cite{T}\label{scattered-bidual}
Let $\mC$ and $\mD$ be two $C^*$-algebras and one of them be scattered. Then,
\[
(\mC \omin \mD)^{**}\cong \mC^{**} \bar{\otimes} \mD^{**},
\]
where $\bar{\otimes}$ denotes the $W^*$-tensor product. \end{theorem}

\begin{proposition}\label{scattered-relations}
Let $\mC$ be a $C^*$-algebra and $\mD$ be a scattered
$C^*$-algebra. Then, for any two $C^*$-subalgebras $\mA$ and $\mB$ of
$\mC$, 
\[
d_{0}(\mA,\mB)\leq d_{0}(\mA \omin \mD, \mB\omin \mD).
\]

In particular, if $\mD$ is commutative as well, then
\[
d_{0}(\mA,\mB) = d_{0}(\mA \omin \mD, \mB\omin \mD).
\]
\end{proposition}

\begin{proof}
Let $\epsilon>0$ and fix a $\gamma_{0}>0$ such that
\[
d_{0}(\mA \omin \mD, \mB\omin \mD)< \gamma_{0}< d_{0}(\mA \omin \mD,
\mB\omin \mD)+ \epsilon.
\]
This implies that $\mA \omin \mD\subseteq_{\gamma_{0}}\mB\omin \mD$
and $\mB \omin \mD\subseteq_{\gamma_{0}}\mA\omin \mD$. We shall show
that $\mA \subseteq_{\gamma_{0}}\mB$ and $\mB
\subseteq_{\gamma_{0}}\mA$.

Since $\mA \omin \mD\subseteq_{\gamma_{0}}\mB\omin \mD$, it follows
from \Cref{enveloping-distance} that $(\mA \omin
\mD)^{**}\subseteq_{\gamma_{0}}(\mB\omin \mD)^{**}$. Further, as $\mD$
is scattered, $\mA^{**} \bar{\otimes}
\mD^{**}\subseteq_{\gamma_{0}}\mB^{**}\bar{\otimes} \mD^{**}$, by
\Cref{scattered-bidual}. We assert that
$\mA^{**}\subseteq_{\gamma_{0}}\mB^{**}$.

Let $w \in B_{1}(\mA^{**})$.  Then, $ w
\otimes 1 \in B_{1}(\mA^{**} \bar{\otimes} \mD^{**})$ and, as
$\mA^{**} \bar{\otimes}
\mD^{**}\subseteq_{\gamma_{0}}\mB^{**}\bar{\otimes} \mD^{**}$, there
exists a $z \in \mB^{**}\bar{\otimes} \mD^{**}$ such that $\|w \otimes
1 - z \|\leq \gamma_{0}.$ Let
$\varphi$ be a normal state on $\mD^{**}$. Then, the slice map
$L_\varphi: \mC^{**} \bar{\otimes} \mD^{**} \to
\mC^{**}$ is a normal conditional expectation - see
\cite[III.2.2.6.]{BL}. Clearly, $L_\varphi$ maps $\mB^{**}
\bar{\otimes} \mD^{**}$ onto $\mB^{**}$; so, $v :=
L_\varphi (z)\in
\mB^{**}$ and
\[
\|w -v\| = \|w \otimes 1 - v \otimes 1 \| = \|L_\varphi(w \otimes 1 - z) \|\leq \|w \otimes 1 - z \|\leq \gamma_{0}.
\]
This proves our assertion. Thus,   $\mA
\subseteq_{\gamma_{0}}\mB$,  by \Cref{enveloping-distance} again. Similiarly, we can conclude that $\mB
\subseteq_{\gamma_{0}}\mA$. Hence, 
\[
d_{0}(\mA,\mB)\leq \gamma_{0}< d_{0}(\mA \omin \mD, \mB\omin \mD)+ \epsilon.
\]
Since $\epsilon>0 $ was arbitrary, 
\[
d_{0}(\mA,\mB)\leq d_{0}(\mA \omin \mD, \mB\omin \mD).
\]

\end{proof}

\end{document}